\documentclass[10pt,a4paper, oneside]{article}
\usepackage{amsfonts}

\usepackage{amssymb,amsthm, amsmath,graphicx,bm,ebezier,amsthm,mathtools}
\usepackage{color,enumerate}
\usepackage{url,lineno}

\usepackage[latin1]{inputenc}

\newtheorem{theorem}{Theorem}
\newtheorem{definition}[theorem]{Definition}
\newtheorem{proposition}[theorem]{Proposition}
\newtheorem{corollary}[theorem]{Corollary}
\newtheorem{lemma}[theorem]{Lemma}

\theoremstyle{remark}

\newtheorem{example}[theorem]{Example}
\newtheorem{algorithm}[theorem]{Algorithm}
\newtheorem{remark}[theorem]{Remark}

\newcommand{\gh}{\widehat{\gamma}}

\def\ow{{\overline{\omega}}}
\def\w{{\omega}}
\def\gcd{\mathrm{gcd}}
\def\lcm{\mathrm{lcm}}
\def\A{\mathcal{A}}

\def\minimals{\mathrm{Minimals}_\leq\,}

\def\E{\mathrm{E}}

\def\N{\mathbb{N}}
\def\R{\mathbb{R}}
\def\Z{\mathbb{Z}}
\def\Q{\mathbb{Q}}

\def\int{\mathrm{int}}

\renewcommand{\parallel}{\|}

\title{Computation of  the $\w$-primality and  asymptotic $\w$-primality with applications to numerical semigroups}

\author{J. I. Garc\'{\i}a-Garc\'{\i}a\footnote{Departamento de Matem\'aticas, Universidad de C\'adiz,
E-11510 Puerto Real (C\'{a}diz, Spain). E-mail: ignacio.garcia@uca.es. Partially supported by the grant MTM2010-15595 and Junta de Andaluc\'{\i}a group FQM-366.}\\
M.A. Moreno-Fr\'{\i}as\footnote{Departamento de Matem\'aticas, Universidad de C\'adiz,
E-11510 Puerto Real (C\'{a}diz, Spain). E-mail: mariangeles.moreno@uca.es. Partially supported by MTM2008-06201-C02-02 and Junta de Andaluc\'{\i}a group FQM-298.}\\
A. Vigneron-Tenorio\footnote{Departamento de Matem\'aticas, Universidad de C\'adiz,
E-11405 Jerez de la Frontera (C\'{a}diz, Spain). E-mail: alberto.vigneron@uca.es. Partially supported by the grant MTM2007-64704 (with the help of FEDER Program), MTM2012-36917-C03-01 and Junta de Andaluc\'{\i}a group FQM-366.}\\
}

\date{}

\begin{document}

\maketitle

\begin{abstract}
We give an algorithm to compute the $\omega$-primality of finitely generated atomic monoids. Asymptotic $\w$-primality is also studied and a formula to obtain it  in finitely generated quasi-Archimedean monoids is proven. The formulation is applied to numerical semigroups, obtaining an expression of this invariant in terms of its system of generators.


\smallskip
{\small \emph{MSC-class:} 20M14,  13A05 (Primaries), 13F15, 20M05 (Secondaries).}
\end{abstract}

\section*{Introduction}
All semigroups appearing in this paper are commutative. For this reason, in the sequel we will omit this
adjective. Here, an atomic monoid means a commutative cancellative semigroup with identity element such that every non-unit may be expressed as a sum of finitely many atoms (irreducible elements).

Problems involving non-unique factorizations in atomic monoids and integral domains have gathered much recent attention in the mathematical literature (see for instance
\cite{Geroldinger-Halter} and the references therein).
Let $S$ be a monoid, the $\omega$-invariant, introduced in \cite{Geroldinger}, is a  well-established invariant in the theory of non-unique factorizations, and  appears also in the context of direct-sum decompositions of modules \cite{Diracca}. This invariant essentially measures how far an element of an integral domain or a monoid is from being prime (see \cite{Anderson-Chapman}).
In \cite{Garcia-Ojeda-Navarro} it is proven that the tame degree and $\omega$-primality coincide for half-factorial affine semigroups and in  \cite{Blanco-Sanchez-Geroldinger} and \cite{Garcia-Ojeda-Navarro} the $\w$-primality is computed for some kinds of affine semigroups (when the semigroup is the intersection of a group and $\Z^p$). Associated with the $\w$-primality there is its asymptotic version, the asymptotic
$\omega$-primality or $\ow$-primality, which is  object of study in several works. In \cite{Chapman},  the $\ow$-primality is  studied for numerical semigroups generated by two elements and it is  given a formula for its computation, but no other studies provide methods to calculate this invariant in other types of monoids.

Our study uses that every finitely generated commutative monoid is finitely presented (see \cite{Redei}). Thus, every monoid can be completely described in terms of a finite set of relations and in case it is cancellative, the monoid is described from a subgroup of $\Z^p$. The first goal of this work is to give an algorithm to compute from a presentation of a finitely generated atomic monoid
the $\omega$-primality of any of its elements. Our second goal is for finitely generated quasi-Archimedean cancellative monoids (note that every numerical semigroup belongs to this class of monoids). For them we give  an explicit formulation of the asymptotic $\w$-primality of its elements, and  a method to compute their asymptotic $\omega$-primalities.

All the theoretical results of this work are complemented with the software {\tt OmegaPrimality} developed in {\tt Mathematica} (see \cite{programa_w-primality}).
This software provides functions to compute the $\omega$-primality of a monoid and its elements from one of its presentations or from a system of generators.

The contents of this paper are organized as follows.
In Section \ref{sec:preli}, we provide some basic tools and definitions that are used in the rest of the work. In Section \ref{sec:w}, we recall the definitions of atomic monoid and   $\w$-primality  and
we give an algorithm to compute  the $\w$-primality of an element.
In Section \ref{sec:2gen}, we start by studying the $\w$-primality and the $\ow$-primality in atomic monoids minimally generated by two elements. Finally, Section \ref{sec:archi} is devoted to give a formulation of the $\ow$-primality in quasi-Archimedean atomic monoids. We finish obtaining the formulation of this invariant for numerical semigroups.

\section{Preliminaries}\label{sec:preli}

Every finitely generated monoid $S$ is isomorphic to a quotient of the form $\N^p/\sigma$ with $\sigma$ a congruence on $\N^p$.  Thus, if $a\in S$, it can be written as $a=[\gamma]_\sigma$ with $\gamma\in\N^p$, where $[\gamma]_\sigma$ denotes the class of equivalence of $\gamma$.
We denote as $\{e_1,\dots ,e_p\}$ the minimal generating set of the monoid $\N ^p$ and
denote by $\parallel(\delta_1,\dots,\delta_p)\parallel=\sum_{i=1}^p\delta_i$
the length of $(\delta_1,\dots,\delta_p)\in\N^p$.
Given $\lambda=(\lambda_1,\dots,\lambda_p),\mu=(\mu_1,\dots,\mu_p)\in\Z^p$, we define the usual cartesian product order $\leq$ on $\Z^p$ as $\lambda\leq \mu$ if and only if
$\mu-\lambda\in\N^p$.
If $A\subseteq\N^p$, define $\minimals A$ as the set of the minimal elements with respect to the order $\leq$. Given two elements $\lambda,\mu\in\N^p$, define $\lambda\vee\mu=(\max\{\lambda_1,\mu_1\},\dots,\max\{\lambda_p,\mu_p\}).$ For any $L$ subset of $\R^p$, denote by $L_>$ the set $\{(x_1,\dots ,x_p)\in L|x_i>0,\,  i=1,\dots,p \}$ and by $L_{\ge}$ the set $\{(x_1,\dots ,x_p)\in L|x_i\ge 0,\,  i=1,\dots,p \}$.

Let $a,b\in S$. We say that $a$ divides $b$ when there exists $c\in S$ such that $a+c=b$, we denote it by $a|b$. The elements $a,b\in S$ are associated if $a|b$ and $b|a$.
When an element $a\in S$ verifies that there exists $b\in S$ such that $a+b=0$, then it is called a unit;
the set of units of $S$ is denoted by $S^\times$. An element $x\in S$ is an atom if it fulfills that $x\not \in S^\times$ and if $a|x$, then either $a\in S^\times$ or $a$ and $x$ are associated.
If the semigroup $S\setminus S^\times$ is
generated by its set of atoms $\A(S)$, the monoid $S$ is called an atomic monoid. It is known that every non-group finitely generated cancellative monoid is atomic (see  \cite[Corollary 16]{atomicos}).

A subset $I$ of a monoid $S$ is an ideal if  $I+S\subseteq I.$ It is straightforward to prove that for every $a\in S$ the set $a+S=\{a+c\,|\,c\in S\}=\{s\in S\,|\,a\textrm{ divides }s\}$ is an ideal of $S$.

\section{Computing the $\omega$-primality in atomic monoids}\label{sec:w}

In this section we show an algorithm to compute the $\w$-primality of an element in a finitely generated atomic monoid from one of its presentations.
The knowledge of the computation of $\w(a)$ for every atom $a$ can be directly used to obtain $\w(S)$ (see \cite[Definition 1.1]{Anderson-Chapman}).
There exist algorithms to compute $\w$ in some kinds of monoids,
for instance in numerical monoids (see \cite{Chapman} and \cite[Remark 5.9.1]{Blanco-Sanchez-Geroldinger}), in half-factorial affine semigroups (see \cite{Garcia-Ojeda-Navarro}), in saturated affine semigroups (see \cite[Corollary 3.5]{Blanco-Sanchez-Geroldinger}), but there is not a general method for its computation in more general situations.

\begin{definition} (See Definition 1.1 in \cite{Chapman}.)
Let $S$ be an atomic monoid with set of units $S^\times$  and set of irreducibles $\A(S)$.
For $s\in S\setminus S^\times$, we define $\omega(x)=n$ if $n$ is the smallest positive integer with the property that whenever $x|a_1+\dots+a_t$, where each $a_i\in\A(S)$, there is a $T\subseteq\{1,2,\dots,t\}$ with $|T|\leq n$ such that $x|\sum_{k\in T}a_k$. If no such $n$ exists, then $\omega(s)=\infty$. For $x\in S^\times$, we define $\omega(x)=0$.
\end{definition}

Recall that every finitely generated monoid $S$ is isomorphic to $\N^p/\sigma$, for some congruence $\sigma$ on $\N^p$ and some positive integer $p$. For sake of simplicity, we will identify $S$ with $\N^p/\sigma$.
Denote by $\varphi:\N^p\to \N^p/\sigma$  the projection map.
For every $A\subset \N^p/\sigma$ denote by $\E(A)$ the set $\varphi^{-1}(A)$. Note that for every $a\in S$ the set $\E(a+S)$ is an ideal of $\N^p$.

The following result is proven in \cite[Proposition 3.3]{Blanco-Sanchez-Geroldinger}
for atomic monoids.

\begin{proposition}\label{p:wminimales}
Let $S=\N^p/\sigma$ be a finitely generated atomic monoid and $a\in S$.
Then
$\omega(a)$ is equal to
$\max \{ \parallel\delta\parallel : \delta\in \minimals (\E( a+S ))\}
.$
\end{proposition}

The set $a+S$ collects all the multiples of $a$, and therefore in many cases it is not a finite set.
The problem now is to compute its minimal elements with respect $\leq$ the usual cartesian product order on $\N^p$.
A solution is given by Algorithm 16 of \cite{irreducibles} which computes the set
$\minimals \E(I)$ for every ideal $I$ of $S$.
The following algorithm  computes the $\w$-primality of an element of an atomic monoid.

\begin{algorithm}\label{algoritmo:w}
Input: A finite presentation of $S=\N^p/\sigma$ and $\gamma$ an element of $\N^p$ verifying that $a=[\gamma]_\sigma$.\\
Output: $\omega(a)$.
\begin{enumerate}[(Step 1)]
\item
Compute the set $\Delta=\minimals ( \E([\gamma]_{\sigma}+S) )$ using \cite[Algorithm 16]{irreducibles}.
\item Set $\Psi=\{\parallel\mu\parallel:\mu\in\Delta\}$.
\item Return $\max \Psi$.
\end{enumerate}
\end{algorithm}

We illustrate now Algorithm \ref{algoritmo:w} with some examples (all the computations were done in an Intel Core i7 with 16 GB of main memory). Note that the inputs in each example are the minimal generator set of the semigroup and an element of the semigroup. The presentations and the expressions of the elements in terms of the atoms are computed internally by the program.

\begin{example}
Let $S$ be the affine semigroup generated by $\{(5,3),(5,11),(2,7),(11,4)\}$.  To compute the $\w$-primality we use the package {\tt OmegaPrimality}
developed for this work, which is available in \cite{programa_w-primality}.
For the element $(154,118)$ we obtained the  output
\begin{verbatim}
In[1]:= OmegaPrimalityOfElemAffSG[{154,118},
                                {{5,3}, {5,11}, {2,7}, {11,4}}]
The expression of the element is {3,5,2,10}
Length of output of Alg16=40
Out[1]= 68
\end{verbatim}
In this case, an expression of $(154,118)$ is $3(5,3)+5(5,11)+2(2,7)+10(11,4)$, the size of $\minimals (\E([\gamma]_{\sigma}+S))$ (the output of \cite[Algorithm 16]{irreducibles}) is equal to 40 and the $\w$-primality of the element is $68$.
\end{example}

In order to compare the timings obtained using Algorithm \ref{algoritmo:w} and the {\tt numericalsgps GAP} package \cite{gapns}, which uses the method described in \cite{Blanco-Sanchez-Geroldinger}, we consider the following examples.  Note that the implementation of $\omega$-primality in the {\tt numericalsgps} relies on the construction of Ap\'{e}ry set while our implementation is based on Gr\"{o}bner basis calculations.
In any case, the procedures in the {\tt numericalsgps} package only apply to numerical semigroups, while the algorithm presented here applies also for any affine semigroup.

\begin{example}\label{ex5}
Let $S$ be the numerical semigroup generated by $\{115,212,333,571\}$. We use again the package {\tt OmegaPrimality} obtaining for the element $10000$ the output
\begin{verbatim}
In[1]:= OmegaPrimalityOfElemAffSG[{10000},
                                {{115}, {212}, {333}, {571}}]
The expression of the element is {3,2,2,15}
Length of output of Alg16=203
Out[1]= 109
\end{verbatim}
The timing
was approximately $22$ seconds and the value obtained for $\w(10000)$ was equal to $109$.
Using the function {\tt OmegaPrimalityOfElementOfNumericalSemigroup} of {\tt numericalsgps GAP} package we obtained
\begin{verbatim}
gap> OmegaPrimalityOfElementInNumericalSemigroup(10000,S);
109
\end{verbatim}
with a timing of approximately $1389$ seconds (and the same value for $\w(10000)$).
\end{example}

\begin{example}
The $\omega$-primality of an atomic monoid $S$ is defined as
$\omega(S)=\sup\{\omega(x)\,|\,x\textrm{ is irreducible}\}$.
Hence,  Algorithm \ref{algoritmo:w} can be used to compute $\w(S)$ when $S$ is atomic and finitely generated just computing the $\omega$-primality of its generators and taking the maximum.
This is implemented in function {\tt OmegaPrimalityOfAffSG}:
\begin{verbatim}
In[2]:= OmegaPrimalityOfAffSG[{{115}, {212}, {333}, {571}}]
...
w-primalities of the generators: {15,36,36,36}
Out[2]= 36
\end{verbatim}
In this case, the $\w$-primalities of the generators are $15$, $36$, $36$ and $36$, respectively, and this computation took 496 milliseconds. Via the {\tt numericalsgps} functions
\begin{verbatim}
gap> OmegaPrimalityOfNumericalSemigroup(S);
36
\end{verbatim}
it took 1888 milliseconds.
\end{example}

\begin{example}
Consider the numerical semigroup $S$ minimally generated by $\{10,\dots ,19\}$.
The package {\tt OmegaPrimality} needed approximately 3779 milliseconds for computing $\omega(S),$
\begin{verbatim}
In[1]:= OmegaPrimalityOfAffSG[{{10},{11},{12},{13},
                                {14},{15},{16},{17},{18},{19}}]
...
w-primalities of the generators: {2,3,3,3,3,3,3,3,3,3}
Out[1]= 3
\end{verbatim}
while the {\tt numericalsgps}
functions took 125 milliseconds.
\end{example}

\begin{example}
Consider now $S=\langle 101, 111, 121, 131, 141, 151, 161, 171, 181, 191\rangle$, its $\omega-$primality was computed in approximately 135081 milliseconds,
\begin{verbatim}
In[1]:= OmegaPrimalityOfAffSG[{{101}, {111}, {121}, {131},
                    {141}, {151}, {161}, {171}, {181}, {191}}]
...
w-primalities of the generators: {12,23,22,22,22,22,22,22,22,22}
Out[1]= 23
\end{verbatim}
and the {\tt numericalsgps} functions took just 383949 milliseconds.
\end{example}

After some comparisons between the times required to obtain the $\omega$-primality using our implementation and {\tt numericalsgps}, we can conclude that the larger are the elements or generators, the better performance one gets with our procedure.  But, if there are many generators and {\em small}, then one should use the Ap\'{e}ry method.

\section{Asymptotic $\omega$-primality in monoids generated by two elements}\label{sec:2gen}

Let $S$ be an atomic monoid and $x\in S$,  define $\overline{\omega}(x)= \lim_{n \to + \infty} \frac{\omega(nx)}{n}$ the asymptotic $\omega$-primality of $x$. Asymptotic $\omega$-primality of $S$ is defined as $\overline\omega(S)=\sup\{\overline\omega(x)|x\textrm{ is irreducible}\}$ (see \cite{Anderson-Chapman}). As in the preceding section  if $S=\langle s_1,\dots,s_p\rangle$, then $\overline\omega(S)$ is equal to $\max\{\ow(s_i)|i=1,\dots,p\}$.

In this section we focus our attention on cancellative reduced monoids minimally generated by two elements.
These monoids are all isomorphic to monoids of the form $\N^2/\sigma$, and they are atomic (see \cite{atomicos}). The following result shows how their congruences are. Note that in particular all numerical monoids generated by two elements are monoids of this type.

\begin{lemma}
A non-free monoid $S$ is cancellative, reduced and minimally generated by two elements if and only if $S\cong\N^2/\sigma$ with $\sigma=\langle((\alpha,0),(0,\beta))\rangle$ and $\alpha,\beta>1$.
\end{lemma}
\begin{proof}
Assume that $S$ is minimally generated by two elements. There exists a congruence $\sigma$ such that $S\cong\N^2/\sigma$.
From Propositions 3.6 and 3.10
in \cite{Rosales3}, there exist positive integers $a$ and $b$ verifying that if $(x, y)\sigma(x_0, y_0)$ then $a(x-x_0)+b(y-y_0) =0.$ Let $M$ be the subgroup of $\Z^2$ defined by the equation $ax + by = 0.$ Then $M = \langle (\alpha,-\beta)\rangle$ with $\alpha = b/d$ and $\beta = a/d$, with $d = \gcd(a, b)$. Hence $(x - x_0, y - y_0) = k(\alpha,-\beta)$ for some integer $k$,
and thus $(x, y) = (x_0, y_0) + k(\alpha,-\beta)$. From this easily follows that
$\sigma=\langle((\alpha,0),(0,\beta))\rangle$.
Since $S$ is not free, $\sigma$ is not trivial and therefore with $\alpha,\beta\neq 0$. Using now that $S$ is reduced and minimally generated by two elements we obtain that $\alpha,\beta>1$.

Assume now that $S\cong\N^2/\sigma$ with $\sigma=\langle((\alpha,0),(0,\beta))\rangle$ and $\alpha,\beta>1$. Using that $\alpha,\beta>1$ it is straightforward to prove that $S$ is minimally generated by two elements and that $S$ is reduced.
We see that
$S$
is cancellative. For every  $\gamma\in\N^2$,  $\E([\gamma]_\sigma)=\{\gamma+\lambda(\alpha,-\beta)|\lambda\in\Z\}\cap\N^2$.
This implies that for all $\gamma'\in\N^2$, we have  $\gamma~\sigma~ \gamma'$ if and only if  $\gamma-\gamma'$ belongs to the subgroup $G$ of $\Z^2$ generated by $(\alpha,-\beta)$. Hence, the congruence $\sigma$ is defined by a subgroup of $\Z^2$ and therefore $S$ is cancellative (see \cite[Proposition 1.4]{Rosales3}).
\end{proof}

In order to compute the asymptotic $\w$-primality of an element $[\gamma]_\sigma\in S$ we give explicitly $\minimals (\E([\gamma]_\sigma+S))$ and $\w([\gamma]_\sigma)$.
Denote by $\lfloor \frac a b \rfloor$ the integer part of $\frac a b$ with  $a,b\in\N$ and $b\neq 0$.

\begin{lemma}\label{lema:multiplos2}
Let $S=\N^2/\sigma$ with $\sigma=\langle((\alpha,0),(0,\beta))\rangle$ and $\alpha,\beta>1$.
Then for all $\gamma=(\gamma_1,\gamma_2)\in\N^2$,
$
\E([\gamma]_\sigma)=\{\gamma+\lambda(\alpha,-\beta)|\lambda\in\Z,-\lfloor\frac {\gamma_1} \alpha \rfloor
\leq \lambda \leq \lfloor \frac {\gamma_2}\beta\rfloor\},
$
\begin{equation}\label{eq1}
\begin{multlined}
\minimals ( \E([\gamma]_\sigma+S) )\\=
\minimals ( \E([\gamma]_\sigma)  \cup
\{ (0,\gamma_2+(\lfloor \frac {\gamma_1}\alpha\rfloor + 1)\beta),
(\gamma_1+(\lfloor \frac {\gamma_2}\beta\rfloor + 1)\alpha,0)\})
\end{multlined}
\end{equation}
and
$\w([\gamma]_\sigma)={\rm max}\{ \gamma_2+(\lfloor \frac {\gamma_1}\alpha\rfloor + 1)\beta, \gamma_1+(\lfloor \frac {\gamma_2}\beta\rfloor + 1)\alpha \}$.
\end{lemma}
\begin{proof}
For every element $\gamma$ the set $\E([\gamma]_\sigma)$ is equal to $\{\gamma+\lambda (\alpha,-\beta)|\lambda\in \Z \}\cap \N^2$. Note that $\gamma+\lambda (\alpha,-\beta)\in\N^2$ if and only if $-\lfloor\frac {\gamma_1} \alpha \rfloor
\leq \lambda \leq \lfloor \frac {\gamma_2}\beta\rfloor$.

 Using now that for every element $z$ in $[\gamma]_\sigma+S$ there exists $\delta\in\N^2$ such that $z=[\gamma+\delta]_\sigma$, we obtain  $\E([\gamma]_\sigma+S)=(\{\gamma+\lambda (\alpha,-\beta)|\lambda\in\Z\}+\N^2)\cap\N^2$.
 To describe this set we need to obtain the intersection of $\E([\gamma]_\sigma +S)$ with the axes. These are the points $ (0,\gamma_2+(\lfloor \frac {\gamma_1}\alpha\rfloor + 1)\beta)$ and $(\gamma_1+(\lfloor \frac {\gamma_2}\beta\rfloor + 1)\alpha,0)$. Thus
 $\E([\gamma]_\sigma+S)=(\E([\gamma]_\sigma)+\N^2)\cup
( (0,\gamma_2+(\lfloor \frac {\gamma_1}\alpha\rfloor + 1)\beta)+\N^2)
 \cup
 ((\gamma_1+(\lfloor \frac {\gamma_2}\beta\rfloor + 1)\alpha,0)+\N^2)
 $ and
therefore (\ref{eq1}) is satisfied.

The set
$\{ \parallel x\parallel \in\N^2| x\in\E([\gamma]_\sigma)\}$ is equal to $\{ \gamma_1+\gamma_2+\lambda(\alpha-\beta) | \lambda\in\Z, -\lfloor \frac {\gamma_1}\alpha\rfloor  \leq \lambda\leq \lfloor \frac {\gamma_2}\beta\rfloor \}$. If $\alpha\geq \beta$ the maximum is achieved by  the element $\gamma+\lfloor \frac{\gamma_2} \beta \rfloor (\alpha,-\beta)$.
Using now that $\gamma_2  - \lfloor \frac {\gamma_2}\beta\rfloor \beta \leq \beta $,  we obtain
$\parallel\gamma+\lfloor \frac{\gamma_2} \beta \rfloor (\alpha,-\beta)\parallel
\leq \parallel (\gamma_1+(\lfloor \frac {\gamma_2}\beta\rfloor + 1)\alpha,0)\parallel$.
Analogously, if $\alpha\leq \beta$
it can be proved that
$\parallel\gamma-\lfloor \frac{\gamma_1} \alpha \rfloor (\alpha,-\beta)\parallel
\leq \parallel (0,\gamma_2+(\lfloor \frac {\gamma_1}\alpha\rfloor + 1)\beta)\parallel$.
Hence,
$\w([\gamma]_\sigma)={\rm max}\{ \gamma_2+(\lfloor \frac {\gamma_1}\alpha\rfloor + 1)\beta, \gamma_1+(\lfloor \frac {\gamma_2}\beta\rfloor + 1)\alpha \}$.

\end{proof}

The proof of the above result can be used to compute the $\w$-primality of any element of a numerical monoid generated by two elements. Let $S$ be the numerical monoid generated by $x<y\in \N$ and let $z=a x +b y\in S$ with $(a,b)\in \N^2$.
 Take $\sigma=\langle ((\alpha,0),(0,\beta))\rangle$ with $\frac {\lcm(x,y)}x=\alpha>\beta=\frac {\lcm(x,y)}y$. It verifies that  $S\cong \N^2/\sigma$ and, by Lemma \ref{lema:multiplos2},
 $\w(z)=\w([(a,b)]_\sigma)=||(a,b)+q(\alpha,-\beta)+(\alpha,-r)||$ where $b=\beta q+r$ with $0\leq r<\beta$.

\begin{example}
Let us consider the congruence
$\sigma=\langle ((7,0),(0,5))\rangle$ and take $\gamma=(6,7)$. In this case, the ideal
$\E([\gamma]_\sigma+S)$ is showed in Figure \ref{f1} and it is generated by the elements $\{(0,12),(6,7),(13,2),(20,0)\}$.
\begin{figure}[h]
\centering
\includegraphics[width=.7\textwidth]{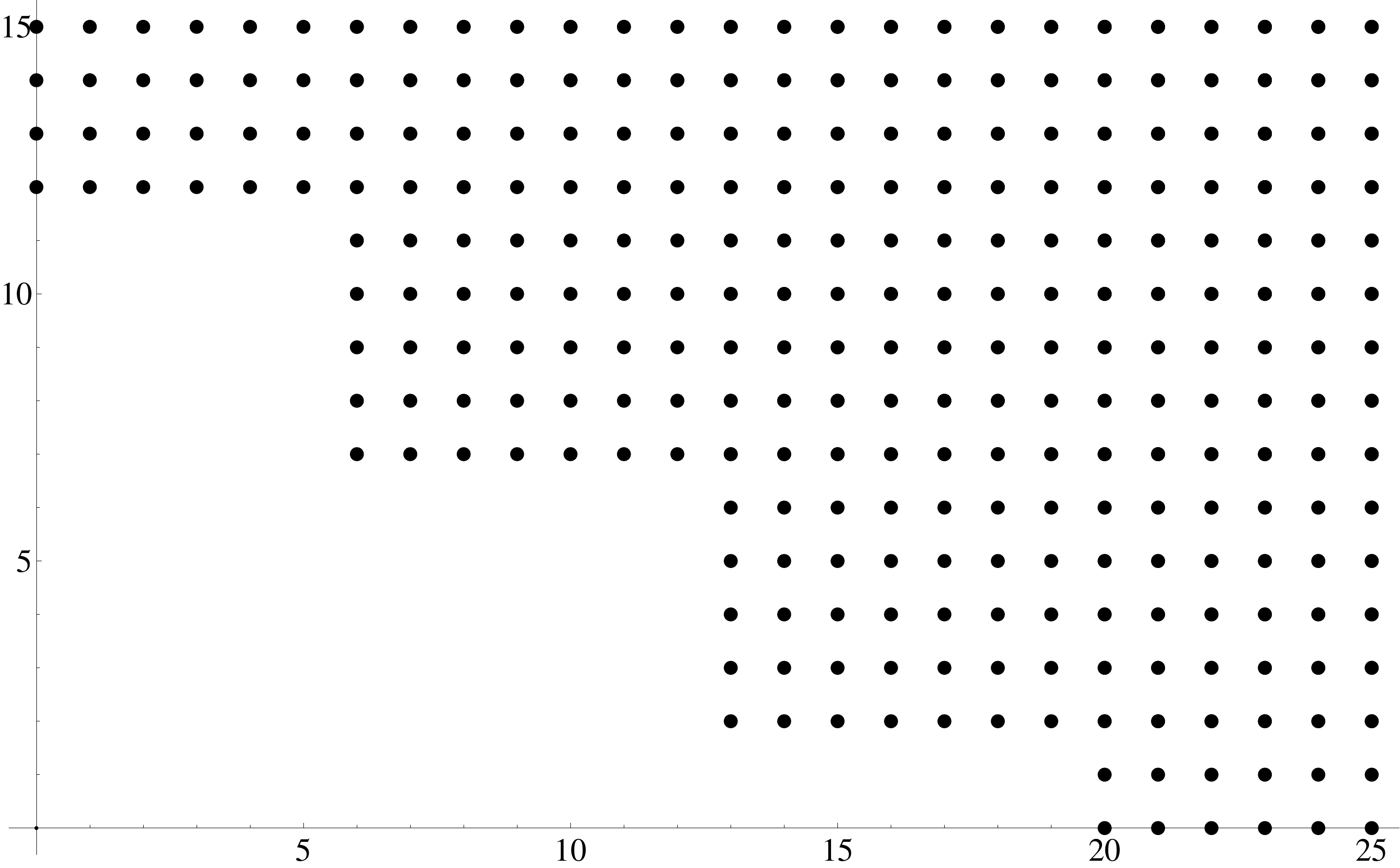}
\caption{$\E([(6,7)]_\sigma+\N^2/\sigma)$.}\label{f1}
\end{figure}
The $\w$-primality of $[(6,7)]_\sigma$ is equal to
$\max\{0+12,6+7,13+2,20+0\}=20$.
\end{example}

\begin{lemma}\label{limites}
For all $a,b\in\N$ with $b\neq 0,$  it  holds that $\lim_{n\to \infty}\lfloor\frac{na}b \rfloor \frac 1 n=\frac a b$.
\end{lemma}
\begin{proof}
For every $n\in\N$, $na$ can be expressed as $na=bm+t$ with $m\in\N$ and $0\leq t<b$. Hence,
$\lfloor\frac{na}b \rfloor \frac 1 n =\lfloor \frac {b m+t} {b} \rfloor\frac a {bm+t}=\lfloor m +\frac {t} {b} \rfloor\frac a {bm+t}= m \frac a {bm+t}\stackrel{m\to \infty}\longrightarrow\frac a{b}.$
\end{proof}

The following result gives us an explicit formulation of the $\ow$-primality
of the elements of the monoids studied in this section.

\begin{proposition}
Let $S=\N^2/\sigma$ with
$\sigma=\langle((\alpha,0),(0,\beta))\rangle$ and $\alpha,\beta>1$.
If
$\alpha\geq\beta$, then  $\ow([(\gamma_1,\gamma_2)]_\sigma)=\gamma_1+\frac{\alpha }{\beta}\gamma_2$ and
if $\alpha<\beta$, then $\ow([(\gamma_1,\gamma_2)]_\sigma)=\frac{\beta }{\alpha}\gamma_1+\gamma_2$.
\end{proposition}
\begin{proof}
Assume that $\alpha\geq \beta$ (if $\alpha < \beta$ proceed similarly). By Lemma \ref{lema:multiplos2},  $\w([\gamma]_\sigma)=\gamma_1+(\lfloor \frac {\gamma_2}\beta\rfloor + 1)\alpha$. Thus, $\ow([\gamma]_\sigma)=\lim_{n\to\infty} \frac {\w(n[\gamma]_\sigma)} n=\lim_{n\to\infty}(n\gamma_1+(\lfloor \frac {n\gamma_2}\beta\rfloor + 1)\alpha)/n$. Using now Lemma \ref{limites}, we have $\ow([\gamma]_\sigma)=\gamma_1+\frac{\alpha \gamma_2}{\beta}$.
\end{proof}

If $\alpha\geq \beta$, since $((\alpha,0),(0,\beta))\in\sigma$ means that
$\alpha[e_1]_\sigma=\beta[e_2]_\sigma$,  abusing of the notation we could say that $\ow([\gamma]_\sigma)=
\frac{[\gamma]_\sigma}{[e_1]_\sigma}=\frac{\gamma_1[e_1]_\sigma+\gamma_2[e_2]_\sigma}{[e_1]_\sigma}=
\gamma_1+ \gamma_2\frac{[e_2]_\sigma}{[e_1]_\sigma}=\gamma_1+\gamma_2\frac \alpha\beta $.

\begin{corollary}
Let $S=\N^2/\sigma$ be an reduced atomic monoid finitely generated by two different elements and let $\sigma=\langle((\alpha,0),(0,\beta))\rangle$ a non-trivial congruence.
 Then:
\begin{itemize}
\item If $\alpha\geq\beta$, then  $\ow([e_1]_\sigma)=1$ and $\ow(S)=\ow([e_2]_\sigma)=\frac \alpha\beta$.
\item If $\alpha<\beta$, then  $\ow([e_2]_\sigma)=1$ and $\ow(S)=\ow([e_1]_\sigma)=\frac \beta\alpha$.
\end{itemize}
Note that  if $\ow([e_i]_\sigma)\neq 1$, the elasticity of $S$ is equal to $\ow([e_i]_\sigma)$ (see \cite{atomicos}).
\end{corollary}

The monoid $S$ is a numerical monoid if and only if $\alpha$ and $\beta$ are coprime. In such  case $S\cong \langle \alpha, \beta\rangle$.

In \cite{Chapman} the asymptotic $\w$-primality and $\w$-primality of numerical monoids generated by two elements is computed.  In that paper it is also given the formulation of the $\w$-primality for  two other families of numerical monoids (monoids generated by $\{n,n+1,\dots ,2n-1\}$ with $n\ge 3$ and by $\{n,n+1,\dots ,2n-2\}$ with $n\ge 4$). The approach proposed in this work allows us to obtain the same formulation for the $\w$-primality of an element, but it can be used in a more general scope. For example, consider  $S$ to be the non-torsion free monoid given by the presentation $\left\{((4,0),(0,2))\right\}$. Since $4$ and $2$ are not coprime, $S$ is not a numerical monoid.
In this case, we have  $\ow([e_1]_\sigma)=1$ and $\ow(S)=\ow([e_2]_\sigma)=\frac{4}{2}=2$.

\section{Asymptotic $\w$-primality in Archimedean semigroups}\label{sec:archi}

We start by introducing some basic concepts used in this section.
An element $x\neq 0$ of a monoid $S$ is Archimedean if for all $y\in S\setminus\{0\}$ there exists a positive integer $k$ such that $y|kx$.
We say
that $S$ is quasi-Archimedean if the zero element is not Archimedean and the rest of elements
in $S$ are Archimedean.
If a monoid is finitely generated, cancellative and quasi-Archimedean, then it verifies that
for all $x,y \in S\setminus\{0\}$, there exist positive integers $p$ and $q$ such that $px=qy$
(see \cite[Theorem 2.1]{MJ} or  \cite{Levin}).
All monoids in Section \ref{sec:2gen}, the submonoids of $\N$ and in particular numerical semigroups are quasi-Archimedean.

\begin{remark}\label{r:ks}
If $S$ is a quasi-Archimedean cancellative monoid, for all $i=2,\dots,p$ there exist $a_i,b_i\in\N\setminus\{0\}$ such that $a_i[e_1]_\sigma=b_i[e_i]_\sigma$. Consider now $a_2\dots a_p [e_1]_\sigma$, applying the above equalities we obtain that $a_2\dots a_p [e_1]_\sigma=b_2 a_3\dots a_p[e_2]_\sigma=\dots=
a_2\dots a_{p-1}b_p[e_p]_\sigma$. For instance if $S$ is the numerical semigroup generated by $5$, $7$ and $11$ we can take $a_1=7$, $b_1=5$, $a_2=11$ and $b_2=5$, obtaining that $77\cdot 5=55\cdot 7=35\cdot 11$.
\end{remark}

Let $S$ be a quasi-Archimedean cancellative monoid. After rearranging its minimal generators we obtain that there exist $k_1\geq \dots \geq k_p\in\N\setminus\{0\}$ verifying that $k_1[e_1]_\sigma=\dots=k_p[e_p]_\sigma$.
In this way some elements of $S$ can be expressed using only the generator $[e_1]_\sigma$.
Abusing again of the notation,
we say that $\frac{[e_i]_\sigma}{[e_1]_\sigma}=\frac {k_1}{k_i}$ and in the same manner that
$\frac {[\gamma]_\sigma}{[e_1]_\sigma}=\sum_{i=1}^p \gamma_i \frac{[e_i]_\sigma}{[e_1]_\sigma}=
\sum_{i=1}^p \gamma_i \frac {k_1}{k_i}$.
Theorem \ref{teorema:ow} proves that $\ow([\gamma]_\sigma)=\frac{[\gamma]_\sigma}{[e_1]_\sigma}$. In order to prove it, we  need two  lemmas.

\begin{lemma}\label{lt1}
Let $k_1\geq\dots\geq k_p\in\N$,  $\gh\in\N^p$,  $\Gamma$ be equal to $\{\gh + \sum_{i,j=1}^p\mu_{ij}(k_ie_i-k_je_j)|\mu_{ij}\in\Z\}$ and let $\Theta=\{\gamma\vee(\gamma+(k_2e_2-k_1e_1))\vee\dots\vee(\gamma+(k_pe_p-k_1e_1)) |\gamma\in\Gamma\}$.
The set $I=\{x\in\Z^p|\textrm{ there exists }\gamma\in\Gamma\textrm{ such that }x\geq \gamma\}$ is equal to $\Z^p\setminus\{x\in\Z^p|\textrm{ there exists }
\phi\in\Theta\textrm{ such that }x_i<\phi_i\textrm{ for all }i=1,\dots,p\}$.
\end{lemma}
\begin{proof}
Let $G$ be the subgroup of $\Z^p$ spanned by $\{k_2e_2-k_1e_1,\dots, k_pe_p-k_1e_1\}$. Then $\Gamma= \gh+G$ and $I=\Gamma+\N^p$. An element $x$ is in $I$ if and only if $x-\gh\in G+\N^p$. Also $a+(b\vee c)=(a+b)\vee(a+c)$ for all $a,b,c\in\Z^p$. For this reason, we may assume without loss of generality that $\gh=0$.

Take $x=(x_1,\dots,x_p)\in\Z^p$. For every $i\in\{2,\dots,p\}$ let $q_i,r_i\in\Z$ such that $x_i=k_iq_i+r_i$, with $0\leq r_i<k_i$ (division algorithm). Set $y=(-k_1\sum_{i=2}^pq_i,q_2k_2,\dots,q_pk_p)$ and $y'=y\vee(y+k_2e_2-k_1e_1)\vee\dots\vee(y+k_pe_p-k_1e_1)$. Then $y\in G$ and $y'=(-k_1\sum_{i=2}^pq_i,(q_2+1)k_2,\dots,(q_p+1)k_p)\in\Theta$. If $-k_1\sum_{i=2}^p q_i\leq x_1$, then $y\leq x$ and thus $x\in G+\N^p$. Otherwise, $x<y'\in\Theta$.
\end{proof}

\begin{lemma}\label{lt2}
Let $S=\N^p/\sigma=\langle s_1,\dots,s_p\rangle$ be a cancellative monoid with $\sigma$ a congruence, let $k_1\geq \dots\geq k_p\in\N$ be such that $k_1s_1=\dots=k_ps_p$ and let $\gamma\in\N^p$. Then every element $x=(x_1,\dots,x_p)\in\N^p\setminus\{0\}$ fulfilling that
\begin{equation}\label{ee1}
\sum_{i=1}^p\frac{k_1\cdots k_p}{k_i}x_i\geq (p-1)k_1\cdots k_p+\sum_{i=1}^p\frac{k_1\cdots k_p}{k_i}\gamma_i
\end{equation}
belongs to $\E([\gamma]_\sigma+S)$.
\end{lemma}
\begin{proof}
Assume that $x\in\N^p$ verifies (\ref{ee1}). Take $\gh=\gamma$ and $\Gamma$, $\Theta$ and $I$ as in Lemma \ref{lt1}. It is easy to prove that for every element in $\Theta$  equality in (\ref{ee1}) holds, and
thus,  by Lemma \ref{lt1},  $x\in I$.
This implies that  there exist $\gamma'\in\Gamma$ and $y\in\N^p$ such that
 $x=\gamma'+y$.
 Since $\gamma'\in\Gamma$, there exist $\mu_{ij}\in\Z$ with $i,j\in\{1,\dots,p\}$ such that $\gamma=\gamma'+\sum_{i,j=1}^p\mu_{ij}(k_ie_i-k_je_j)$.
 Thus, $\gamma+y=x+\sum_{i,j=1}^p\mu_{ij}(k_ie_i-k_je_j)\in\N^p$.
 Since $S$ is cancellative there exists $G$ a subgroup of $\Z^p$ such that for every $a,b\in\N^p$, $a\,\sigma\,b$ if and only if $a-b\in G$ (see \cite[Proposition 1.4]{Rosales3}).
Using this fact and that  $k_ie_i\,\sigma\,k_je_j$, we have that
$\sum_{i,j=1}^p\mu_{ij}(k_ie_i-k_je_j)\in G$ and therefore
$x \,\sigma\,x+\sum_{i,j=1}^p\mu_{ij}(k_ie_i-k_je_j)=\gamma+y$.
 Hence, $x\in\E([\gamma]_\sigma+S)$.
\end{proof}

\begin{theorem}\label{teorema:ow}
Let $S=\N^p/\sigma$
 be a quasi-Archimedean cancellative reduced monoid.
There exists a rearrange $\{t_1,\dots,t_p\}$ of the set $\{1,\dots,p\}$ such that
$\ow(a)=\gamma_{t_1}+\sum_{i=2}^{p} \frac {k_{t_1} \gamma_{t_i}} {k_{t_i}} $
for every $a=[(\gamma_1,\dots,\gamma_p)]_\sigma \in S$.
\end{theorem}
\begin{proof}

By \cite[Lemma 3.3]{Geroldinger-Hassler}, $\w( a+b)\leq \w(a)+\w(b)$ for all $a,b\in S$. Thus, for every $n,m\in\N$ and every $a\in S$, $\w((n+m) a )\leq n\w(a)+m\w(a)$. Fekete's Subadditive Lemma (see \cite{Fekete}) states that for every subadditive sequence $\{z_n|n=1,\dots,\infty\}$, the limit $\lim_{n\to \infty} \frac {z_n} n$ exists and it is equal to ${\inf} \frac {z_n} n$ or $-\infty$. Since $\w(a)\geq 0$ for every $a\in S\setminus\{0\}$, the limit $\ow(a)=\lim_{n\to\infty} \frac {\w(n a)} n$
always exists for all $a\in S$.

Without loss of generality we can assume that there exist $k_1\geq \dots\geq k_p\in\N\setminus\{0\}$ verifying
$k_1[e_1]_\sigma=\dots=k_p[e_p]_\sigma$.
Then for all $\gamma=(\gamma_1,\dots,\gamma_p)\in\N^p$, it holds that
$k_2\cdots k_p[\gamma]_\sigma=
k_2\cdots k_p\gamma_1[e_1]_\sigma+\dots+k_2\cdots k_p\gamma_p[e_p]_\sigma=
(\gamma_1k_2\cdots k_p +k_1\gamma_2k_3\cdots k_p+\dots+ k_1\dots k_{p-1}\gamma_p)[e_1]_\sigma=
[(\sum_{i=1}^p\frac{k_1\cdots k_p} {k_i} \gamma_i,0,\dots,0)]_\sigma
$.

Let $a=[\gamma]_\sigma\in S$ with $\gamma\in\N^p$.
Denote by $H$ the hyperplane of $\Q^p$ spanned by $\{ k_1 e_1 -k_ie_i|i=2,\dots,p \}$; $H$ is defined by the  equation
$\sum_{i=1}^p\frac{k_1\cdots k_p} {k_i} x_i =0$.
It is straightforward to prove that
$\E([\gamma]_{\sigma})\subseteq H_\gamma$ where
 $H_\gamma$ is the affine subvariety of $\Q^p$ defined by the equation $\sum_{i=1}^p\frac{k_1\cdots k_p} {k_i} x_i =\sum_{i=1}^p\frac{k_1\cdots k_p} {k_i} \gamma_i$.
 Using now that $0 \leq \frac{k_1\cdots k_p} {k_1} \leq\dots\leq \frac{k_1\cdots k_p} {k_p}$
 and that $ \sum_{i=1}^p\frac{k_1\cdots k_p} {k_i} \gamma_i \geq 0$ we deduce that
$\max \{ \sum_{i=1}^p x_i| (x_1,\dots,x_p)\in  H_{\gamma}\cap\Q^p_{\geq} \}$  is achieved by the element
$\frac {1}{ k_2\cdots k_p}(\sum_{i=1}^p\frac{k_1\cdots k_p} {k_i} \gamma_i,0,\dots,0)$.
For the elements of the form $n k_2 \cdots k_p [\gamma]_\sigma$ with $n\in\N\setminus\{0\}$ its  maximum is achieved by the element
$n k_2\cdots k_p \frac {1}{ k_2\cdots k_p}(\sum_{i=1}^p\frac{k_1\cdots k_p} {k_i} \gamma_i,0,\dots,0)=
n(\sum_{i=1}^p\frac{k_1\cdots k_p} {k_i} \gamma_i,0,\dots,0)\in \N^p$, which clearly verifies that
\[nk_2\cdots k_p [\gamma]_\sigma = [n(\sum_{i=1}^p\frac{k_1\cdots k_p} {k_i} \gamma_i,0,\dots,0)]_\sigma.\]
Using that $S$ is cancellative and reduced, it is easy to prove that
$n(\sum_{i=1}^p\frac{k_1\cdots k_p} {k_i} \gamma_i,0,\dots,0)
\in \minimals (\E(nk_2\cdots k_p[\gamma]_\sigma+S))
$.
Thus,
$\w(nk_2\cdots k_p[\gamma]_\sigma)\geq
\parallel n (\sum_{i=1}^p\frac{k_1\cdots k_p} {k_i} \gamma_i,0,\dots,0)\parallel=n \sum_{i=1}^p\frac{k_1\cdots k_p} {k_i} \gamma_i$.
Since the limit $\ow(a)=\lim_{n\to \infty} \frac {\w(n [\gamma]_\sigma) } n$  exists and every subsequence  converges to the same limit. Therefore,
\[
\begin{multlined}
\ow(a)=\lim_{n\to \infty} \frac {\w(n [\gamma]_\sigma) } n=
\lim_{n \to \infty} \frac {\w(n k_2\cdots k_p [\gamma]_\sigma) } {nk_2\cdots k_p}\geq
\lim_{n\to\infty} \frac{n \sum_{i=1}^p\frac{k_1\cdots k_p} {k_i} \gamma_i}{n k_2 \cdots k_p} \\=
\lim_{n \to \infty} \frac {nk_1k_2\cdots k_p(\gamma_1/k_1+\dots+\gamma_p /k_p)}{nk_2\cdots k_p}=
k_1 (\gamma_1/k_1+\dots+\gamma_p /k_p).
\end{multlined}
\]
Obtaining in this way a lower bound for $\ow(a)$.

We now prove that the equality holds. By Lemma \ref{lt2}, every element
$(x_1,\dots,x_p)\in\N^p$ verifying that
$\sum_{i=1}^p\frac{k_1\cdots k_p}{k_i}x_i\geq (p-1)k_1\cdots k_p+\sum_{i=1}^p n \frac{k_1\cdots k_p}{k_i}\gamma_i
$ belongs to $\E(n[\gamma]_\sigma +S)$. Thus,
 every element fulfilling
$\sum_{i=1}^p\frac{k_1\cdots k_p}{k_i}(x_i-1)\geq (p-1)k_1\cdots k_p+\sum_{i=1}^p n \frac{k_1\cdots k_p}{k_i}\gamma_i$
also belongs to $\E(n[\gamma]_\sigma +S)$, but
does not belong to
$\minimals (\E(n[\gamma]_\sigma+S))$. Therefore, the elements of
$\minimals (\E(n[\gamma]_\sigma+S))$ satisfy the inequality
\begin{equation}\label{eq3}
\sum_{i=1}^p\frac{k_1\cdots k_p}{k_i}x_i\leq \sum_{i=1}^p\frac{k_1\cdots k_p}{k_i}+
(p-1)k_1\cdots k_p+\sum_{i=1}^p n \frac{k_1\cdots k_p}{k_i}\gamma_i.
\end{equation}
Hence, the set
$
\minimals (\E(n[\gamma]_\sigma+S))$
is included in
$\left\{ (x_1,\dots,x_p)\in\Q^p_\geq :
(x_1,\dots,x_p) \textrm{ satisfies \textrm{(\ref{eq3})}}
\right\}
$ and for this reason
\begin{equation}\label{eq4}
\w(n[\gamma]_\sigma)
\leq
\sup\{ \|(x_1,\dots,x_p)\|
:
(x_1,\dots,x_p)\in\Q^p_\geq\textrm{ and }
 \textrm{ satisfies \textrm{(\ref{eq3})}}
\}
.
\end{equation}
Using  that the coefficients of the left-hand side of  (\ref{eq3}) verify that
$0 \leq \frac{k_1\cdots k_p} {k_1} \leq\dots\leq \frac{k_1\cdots k_p}{k_p}$, it is straightforward to prove that for every $d\in\Q$
the set
$
\{
(x_1,\dots,x_p)\in\Q^p_\geq :
\sum_{i=1}^p\frac{k_1\cdots k_p}{k_i}x_i=d
\}
$
is nonempty if and only if $d\geq0$. Besides,
the maximum of
$
\{ \|(x_1,\dots,x_p)\|
:
(x_1,\dots,x_p)\in\Q^p_\geq\textrm{ and }
\sum_{i=1}^p\frac{k_1\cdots k_p}{k_i}x_i=d
\}
$
 is achieved by
the element $(d/(k_2\cdots k_p),0,\dots,0)$, it is equal to $d/(k_2\cdots k_p)$, and if we increase $d$, then  it increases.
Since
the  right-hand side of  (\ref{eq3}) is greater than zero, we deduce that
the supreme of (\ref{eq4}) is achieved by an element of $\Q^p_\geq$ fulfilling that equality in (\ref{eq3}) holds and
having all its coordinates equal to zero but the first one.
That point is equal
to $n\xi$ with
\[\xi=\left(
\frac {\sum_{i=1}^p\frac{k_1\cdots k_p}{k_i}} {nk_2\cdots k_p} +
\frac {(p-1)k_1} n +
\frac { \sum_{i=1}^p
\frac{k_1\cdots k_p}{k_i}\gamma_i}
{k_2\cdots k_p}
,0,\dots,0\right)\in\Q_{\geq}^p\]
and it verifies that
$\w(n[\gamma]_\sigma)\leq \|n\xi\|$.
Therefore,
\[
\begin{multlined}
\ow(a)=\lim_{n\to \infty} \frac {\w(n [\gamma]_\sigma) } n
\leq \lim_{n\to \infty} \frac{\|n\xi\|}n\\
=\lim_{n\to \infty}
\frac
{
\frac {\sum_{i=1}^p\frac{k_1\cdots k_p}{k_i}} {k_2\cdots k_p} +
\frac {(p-1)k_1} 1 +
n\frac { \sum_{i=1}^p
\frac{k_1\cdots k_p}{k_i}\gamma_i}
{k_2\cdots k_p}
}
n
=
\frac{\sum_{i=1}^p\frac{k_1\cdots k_p} {k_i} \gamma_i}{k_2 \cdots k_p}
\\=
k_1 (\gamma_1/k_1+\dots+\gamma_p /k_p).
\end{multlined}
\]

\vspace{-22 pt}
\qedhere
\vspace{22 pt}
\end{proof}

\begin{corollary}
Let $S=\N^p/\sigma$ be a quasi-Archimedean cancellative reduced monoid.
There exist $k_1,\dots,k_p\in\N$ such that $\ow([e_i]_\sigma)=\frac { \max\{k_1,\dots,k_p  \}  } {k_i}$ for all $i=1,\dots,p$.
\end{corollary}

As we pointed out in the abstract, we give a formula to compute the $\ow$-primality in numerical semigroups.

\begin{corollary}
Let $S$ be a numerical monoid minimally generated by
$\langle s_1<s_2<\dots<s_p\rangle$. For every
$s\in S$,
we have that
$\ow(s)=\frac s {s_1}$.
\end{corollary}
\begin{proof}
Use that there exist $\gamma_1,\dots,\gamma_p\in\N$ such that $s=\sum_{i=1}^p \gamma_i s_i$ and
 take $k_i=\frac {\lcm(s_1,\dots,s_p)} {s_i}$.
\end{proof}

\end{document}